\documentclass[12pt]{amsart}

\usepackage{graphicx,amssymb,enumerate}
\usepackage{hyperref}
\usepackage{color}
\usepackage{nameref,cleveref}
\usepackage{mathdots}
\usepackage{pstricks}
\usepackage[utf8]{inputenc}
\newcommand{\Salem}[4]{S^{(#1)}_{#2,#3,#4}}
\newcommand{\cox}[1]{\operatorname{cox}_{#1}(t)}
\newcommand{\coxwo}[1]{\operatorname{cox}_{#1}(t)}
\newcommand{\cal}[1]{\mathcal{#1}}
\newcommand{\ad}{\operatorname{Ad}}
\makeatletter
\newcommand{\rdots}[1]{%
  \hbox to .2em{\hss {\rotatebox[origin=c]{#1}{$\displaystyle\cdots$}}\hss}}
\newtheorem{thm}{Theorem}[section]
\newtheorem{cor}[thm]{Corollary}
\newtheorem{lem}[thm]{Lemma}
\newtheorem{proposition}[thm]{Proposition}

\theoremstyle{definition}
\newtheorem{defin}[thm]{Definition}
\newtheorem{rem}[thm]{Remark}
\newtheorem{exa}[thm]{Example}

\numberwithin{equation}{section}

\@namedef{subjclassname@2010}{%
  \textup{2010} Mathematics Subject Classification}
\makeatother

\crefname{proposition}{proposition}{propositions}
\Crefname{proposition}{Proposition}{Propositions}
\crefname{thm}{theorem}{theorems}
\Crefname{thm}{Theorem}{Theorems}

\renewcommand{\cref}{\Cref}
\begin{document}
\title{Coxeter polynomials of Salem trees}

\author{Charalampos A. Evripidou}
\email{\vspace*{-1ex} evripidou.charalambos@ucy.ac.cy}
\thanks{This work was co-funded by the European Regional Development Fund and the Re\-pu\-blic of Cyprus through the Research Promotion Foundation (Project: PENEK/0311/30).}

\address{Department of Mathematics and Statistics, University of Cyprus, P.O. Box 20537, 1678 Nicosia, Cyprus}
\date{}

\begin{abstract}
We compute the Coxeter polynomial of a family of Salem trees
and also the limit of the spectral radius of their Coxeter transformations as the
number of their vertices tends to infinity.
We also prove that if $z$ is a root of multiplicities $m_1,\ldots,m_k$ for the Coxeter po\-ly\-no\-mials
of the trees $\cal{T}_1,\ldots,\cal{T}_k$, then $z$ is a root for the Coxeter po\-ly\-no\-mial
of their join, of multiplicity at least $\min\{m-m_1,\ldots,m-m_k\}$ where $m=m_1+\ldots+m_k$.
\end{abstract}

\subjclass[2010]{20F55}

\keywords{Coxeter polynomial; Coxeter transformation; spectral radius; Dynkin diagrams}
\maketitle

\section{Introduction and preliminaries}
In \cite{lakatos}, Lakatos determines the limit of the spectral radii of
the Co\-xe\-ter transformations of particular infinite sequences of starlike trees.
In the present paper we generalize the result of Lakatos \cite{lakatos} to a wider range of trees.
In addition, our idea of proof is different from the one in \cite{lakatos}.

We use the same terminology as in
\cite{lakatos,simson_algorithms} and \cite{simson_a_framework_coauthor}. We denote by
$\mathbb N\subseteq \mathbb Z$ the set of nonnegative integers
and the ring of integers, respectively. The algebra of the $n\times n$ square integer matrices
is denoted by $\mathbb M_n(\mathbb Z)$, where $n\in\mathbb N$.
We consider only simple graphs (i.e., graphs without multiple edges and loops),
$\Gamma=(\Gamma_0,\Gamma_1)$ with the set of vertices $\Gamma_0=\{v_1,\ldots,v_n\}$ and $\Gamma_1$ the set of edges,
where $(v_i,v_j)\in\Gamma_1$ if there is an edge connecting the vertices $v_i$ and $v_j$.

Assume that $\Gamma=(\Gamma_0,\Gamma_1)$ is a simple graph with the set of enumerated vertices $\Gamma_0=\left\{v_1,\ldots,v_n\right\}$.
We recall that the \textbf{adjacency matrix} of the graph $\Gamma$ is the $n\times n$ symmetric matrix
\begin{equation}
\label{Adjacency_matrix}
\ad_{\Gamma}=[a_{ij}]\in\mathbb{M}_n(\mathbb{Z})
\end{equation}
with $a_{ij}=1$, if $(v_i,v_j)\in \Gamma_1$ and $a_{ij}=0$, otherwise.
The \textbf{characteristic polynomial} of $\Gamma$ is defined to be the polynomial 
\begin{equation}
\label{characteristic_polynomial}
\chi_{\Gamma}(t):=\det(t\cdot I_n-\ad_{\Gamma})\in\mathbb{Z}[t]
\end{equation}
where $I_n=[\delta_{ij}]$ is the identity matrix in $\mathbb{M}_n(\mathbb{Z})$.
It is clear that $\chi_{\Gamma}(t)$ does not depend on the enumeration $v_1,\ldots,v_n$
of the vertices in $\Gamma_0$, see \cite{brualdi2011} and \cite{Cvetkovic2010}.

Let $\mathbb{R}^n$ be the standard $n$ dimensional real vector space with the standard basis $e_1,\ldots,e_n$.
Given $i\in\left\{1,\ldots,n\right\}$, the $i$th reflection of $\Gamma$ is defined to be the
$\mathbb{R}$-linear automorphism $\sigma_i:\mathbb{R}^n\rightarrow\mathbb{R}^n$ given by the formula 
\begin{equation}
\label{reflection}
\sigma_i(e_j)=e_j-\left(2\delta_{ij}-a_{ij}\right)e_i.
\end{equation}
The subgroup $W_{\Gamma}$ of the general linear group $GL(\mathbb{R}^n)\cong GL(n,\mathbb{R})$ generated by the reflections
$\sigma_1,\ldots,\sigma_n$ of $\Gamma$ is called the \textbf{Weyl group} of $\Gamma$ and has the presentation 
\begin{equation}
\label{presentation_of_weyl_group}
W_{\Gamma}=\langle\sigma_1,\sigma_2,\ldots,\sigma_n:(\sigma_i\sigma_j)^{m_{ij}}=1\rangle
\end{equation}
where $M=[m_{ij}]\in\mathbb{M}_n(\mathbb{Z})$ is the matrix defined by $m_{ii}=1$ for all $i=1,\ldots,n$
and $m_{ij}=a_{ij}+2$ for all $i\neq j$, see \cite{bourbakifourtosix,humphreyscox,steko}.
The product $\Phi_{\Gamma}=\sigma_1\cdot\ldots\cdot\sigma_n\in W_{\Gamma}$ is defined to be the \textbf{Coxeter transformation} of the graph $\Gamma$, see \cite{delapena1992}.
Obviously, it depends on the enumeration of the vertices $v_1,\ldots,v_n$ of $\Gamma$,
see \cref{Remark:uniqueness_of_coxeter_transformation} for details.
We recall that the Coxeter transformations were first studied by Coxeter in \cite{coxeter}
where he showed that their eigenvalues have remarkable properties,
see also Bourbaki \cite{bourbakifourtosix} and Humphreys \cite{humphreyscox}.

Throughout this paper, we assume that $\Gamma$ is a tree $\cal{T}=(\cal{T}_0,\cal{T}_1)$ with enumerated vertices $\cal{T}_0=\{v_1,\ldots,v_n\}$,
$\ad_{\cal{T}}=[a_{ij}]\in\mathbb{M}_n(\mathbb Z)$ is its adjacency matrix, and
\begin{equation}
\label{Equation:definition_of_coxeter_transformation}
\Phi_{\cal{T}}=\sigma_1\cdot\sigma_2\cdot\ldots\cdot\sigma_n\in W_{\cal{T}}
\end{equation}
is its Coxeter transformation, with respect to the enumeretion $v_1,v_2,\ldots,v_n$.
The Coxeter polynomial of the tree $\cal{T}$ is defined to be the characteristic polynomial of $\Phi_{\cal{T}}:\mathbb{R}^n\rightarrow\mathbb{R}^n$
that is, the polynomial (see \cite{humphreyscox,delapena1992,simson_a_coxeter-gram_classification})
\begin{equation}
\label{coxeter_polynomial}
\cox{\cal{T}}:=\det(t\cdot \operatorname{id}_{\mathbb{R}^n}-\Phi_{\cal{T}})\in\mathbb{Z}[t].
\end{equation}

Since $\cal{T}$ is a tree, the characteristic polynomial of the transformation $\Phi_{\cal{T}}$
does not depend on the enumeration of the vertices $v_1,\ldots,v_n$.
Indeed, if $v_{\epsilon(1)},\ldots,v_{\epsilon(n)}$ is obtained from $v_1,\ldots,v_n$ by a permutation
$\epsilon\in S_n$ then the Coxeter transformation $\Phi_{\cal{T}}^{\epsilon}:\mathbb{R}^n\rightarrow\mathbb{R}^n$
corresponding to the enumeration $v_{\epsilon(1)},\ldots,v_{\epsilon(n)}$ is conjucate with $\Phi_{\cal{T}}$,
see \cite[Proposition 2.2]{simson_a_coxeter-gram_classification}, \cite[Proposition 3.16]{humphreyscox},
\cite{bourbakifourtosix,delapena1992} and the following remark for details.

\begin{rem}
\label{Remark:uniqueness_of_coxeter_transformation}
(a) The Coxeter polynomial $\cox{\Delta}$ is also defined
and stu\-died in \cite{simson_algorithms,simson_a_coxeter-gram_classification} and \cite{simson_a_framework_single}
in a more general setting of loop-free edge-bipartite multigraphs $ \Delta = (\Delta_0,\Delta_1 = \Delta^{-}_1\cup\Delta^{+}_1)$,
with $\Delta_0=\left\{v_1,v_2,\ldots,v_n\right\}$ and a separated bipartition $\Delta_1 = \Delta^{-}_1\cup\Delta^{+}_1$ of the set of edges.
The class of loop-free edge-bipartite multigraphs contains all simple graphs,
loop-free multigraphs, and simple signed graphs, see \cite{zaslavsky_signed_graphs}.

The definition of $\cox{\Delta}\in\mathbb Z[t]$ for an edge-bipartite multigraph $\Delta$,
differs from the one given in (\ref{coxeter_polynomial})
for simple graphs, and depends on the upper triangular Gram matrix
$\check{G}_{\Delta}=[d_{ij}^{\triangle}]\in GL(n,\mathbb Z)$ where $d_{ij}^{\triangle}=1$ for $i=j, d_{ij}^{\triangle}$
is the number of edges between $v_i$ and $v_j$, with $i< j$, lying in $\Delta_1^+$ and $-d_{ij}^{\triangle}$
is the number of edges between $v_i$ and $v_j$, with $i<j$, lying in $\Delta_1^-$.

In \cite{simson_algorithms,simson_a_coxeter-gram_classification}
and \cite{simson_a_framework_single},
with any loop-free edge-bipartite multigraph $\Delta=(\Delta_0,\Delta_1=\Delta_1^-\cup\Delta_1^+)$
the Coxeter matrix $\operatorname{Cox}_{\Delta}:=-\check{G}_\Delta\cdot\check{G}_\Delta^{-\operatorname{tr}}\in\mathbb{M}_n(\mathbb{Z})$
is associated and its characteristic polynomial
\begin{equation}
\label{coxeter_polynomial_of_edge-bipartite}
\cox{\Delta}:=\det(t\cdot I_n-\operatorname{Cox}_{\Delta})\in\mathbb{Z}[t],
\end{equation}
called the \textbf{Coxeter polynomial} of $\Delta$ is self-reciprocal in the sense that  $\cox{\Delta}=t^n\operatorname{cox}_{\Delta}\left(\frac{1}{t}\right)$,
see Lemma 2.8 (c3)-(c4) in \cite{simson_mesh_geometries}.
The \textbf{Coxeter tra\-nsfo\-rma\-tion} of $\Delta$ is defined to be the group automorphism
\begin{equation}
\label{coxeter_tranformation}
\Phi_{\Delta}:\mathbb{Z}^n\rightarrow\mathbb{Z}^n,v\mapsto v\cdot\operatorname{Cox}_{\Delta}.
\end{equation}
It is proved in \cite[Proposition 2.2]{simson_a_coxeter-gram_classification} that in the case when the underlying multigraph
$\overline{\Delta}$ of $\Delta$ is a tree, the Coxeter polynomial does not depend on the enumeration of the vertices $v_1,\ldots,v_n$.
Hence, in view of the sink-source reflection technique applied in
\cite[Proposition VII.4.7]{simson_elements}, the Coxeter polynomial
$\cox{\Delta}$ (\ref{coxeter_polynomial_of_edge-bipartite}) of $\Delta$ coincides with the Coxeter polynomial
$\cox{\overline{\Delta}}$ of the tree $\cal{T}=\overline{\Delta}$ (in the sense of (\ref{coxeter_polynomial})).

The reader is also referred to the recent papers
\cite{simson_and_kasjan_mesh_algorithms_first,simson_and_kasjan_mesh_algorithms_second},
where the irreducible and reduced root systems in the sense of Bourbaki \cite{bourbakifourtosix}
are studied in connection with roots of positive connected edge-bipartite graphs.

(b) The Coxeter polynomial is also defined in \cite{simson_integral_bilinear_forms} and \cite{simson_a_framework_coauthor},
for any finite poset $J\equiv(J,\preceq)$, with $J=\{1,\ldots,n\}$, as
\begin{equation}
\label{coxeter_polynomial_of_poset}
\cox{J}:=\det(t\cdot I_n-\operatorname{Cox}_{J})\in\mathbb{Z}[t]
\end{equation}
where $\operatorname{Cox}_{J}=-C_J\cdot C_J^{-\operatorname{tr}}\in\mathbb{M}_n(\mathbb{Z})$
is the Coxeter matrix of $J$ and $C_J:=[c_{ij}]\in\mathbb{M}(\mathbb{Z})$ is its incidence matrix,
with $c_{ij}=1$, for $i\preceq j$, and $c_{ij}=0$ if $i\not \preceq j$.
It is shown that if the Hasse diagram $H:=\cal{H}_J$ of $J$ is a tree,
then the Coxeter polynomial $\cox{J}$ (\ref{coxeter_polynomial_of_poset}) of $J$
coincides with the Coxeter polynomial $\cox{H}$ of the tree $\cal{T}=H$ (in the sense of (\ref{coxeter_polynomial})).
\end{rem}

By applying \cref{Remark:uniqueness_of_coxeter_transformation}(a) we get the following useful fact

\begin{cor}
\label{Corollary}
Assume that $\cal{T}=(\cal{T}_0,\cal{T}_1)$ is a tree with enumerated vertices $v_1,\ldots,v_n$
and let $\check{G}_{\cal{T}}=[d_{ij}]\in\mathbb{M}_n(\mathbb{Z})$ be the upper triangular Gram matrix
of $\cal{T}$, with $d_{11}=\ldots=d_{nn}=1, d_{ij}=-1$ if $i<j$ and there is an edge $(v_i,v_j)$ in $\cal{T}_1$
and $[d_{ij}]=0$, otherwise.

(a) The Coxeter transformation $\Phi_{\cal{T}}:\mathbb{R}^n\rightarrow\mathbb{R}^n$
(\ref{Equation:definition_of_coxeter_transformation}) of the tree $\cal{T}$
restricts to the group automorphism $\Phi_{\cal{T}}:\mathbb{Z}^n\rightarrow\mathbb{Z}^n$ defined by the formula
$$
\Phi_{\cal{T}}(u)=u\cdot \operatorname{Cox}_{\cal{T}}
$$
where $\operatorname{Cox}_{\cal{T}}:=-\check{G}_{\cal{T}}\cdot\check{G}_{\cal{T}}^{-\operatorname{tr}}\in\mathbb{M}_n(\mathbb{Z})$
is the Coxeter matrix of $\cal{T}$ viewed as an edge-bipartite graph, with  $\cal{T}_1^+$ empty.

(b) The Coxeter polynomial $\cox{\cal{T}}$ (\ref{coxeter_polynomial}) of the tree $\cal{T}$
coincides with the Coxeter polynoial $\cox{\mathcal{T}}=\det(t\cdot{I_n}-\operatorname{Cox}_\Delta)$
(\ref{coxeter_polynomial_of_edge-bipartite}) of $\cal{T}$ viewed as an edge-bipartite tree.
 
(c) The Coxeter polynomial $\cox{\cal{T}}$ (\ref{coxeter_polynomial})
of the tree $\cal{T}$ is self-reciprocal and does not depend on the enumeration of the vertices $v_1,\ldots,v_n$ of the tree $\cal{T}$.
\end{cor}

\begin{proof}

We view $\cal{T}$ as an edge-bipartite graph, with $\cal{T}_1=\cal{T}_1^-\cup\cal{T}_1^+$
where $\cal{T}_1^+$ is the empty set. Then the matrix $\check{G}=[d_{ij}]\in\mathbb{M}_n(\mathbb Z)$ coincides with the upper triangular Gram matrix
$\check{G}_{\Delta}=[a_{ij}^{\triangle}]$ defined in \cref{Remark:uniqueness_of_coxeter_transformation}(a).
Then the corollary is a consequence of \cref{Remark:uniqueness_of_coxeter_transformation}(a).
\end{proof}

The most important families of trees are the trees of type $ADE$
given in \Cref{Fig:ADE_graphs}.
These trees are known as the simply laced Dynkin diagrams.
There is a long list of objects which admit an $ADE$ classification, meaning that there is an equivalence
between equivalence classes of objects of the given type and the $ADE$ graphs (see for example \cite{hazewinkel}). Examples of these objects
include the
\begin{enumerate}[-]
\item simply laced finite Coxeter groups, 
\item simply laced simple Lie algebras,
\item platonic solids,
\item quivers of finite representation types,
\item Kleinian singularities,
\item finite subgroups of the $\operatorname{SU}(2)$ group.
\end{enumerate}

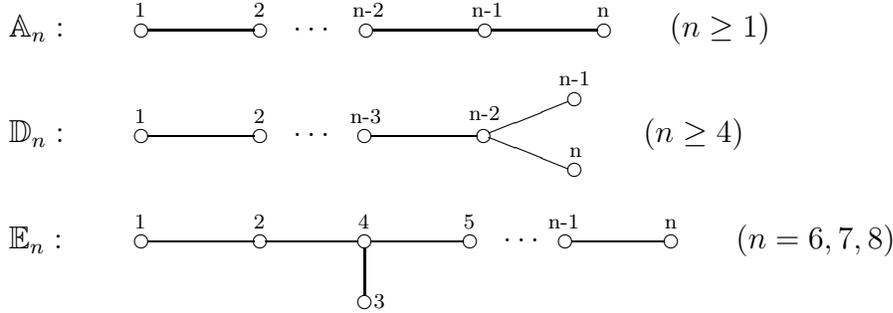
\begin{figure}[ht]
\hspace{4em}
\begin{picture}(355,120)

$
	\put(52.5,100){\line(1,0){40}}
	\put(137.5,100){\line(1,0){40}}
	\put(182.5,100){\line(1,0){40}}
	\put(50,100){\circle{5}}
	\put(47.5,105){{\tiny 1}}
	\put(95,100){\circle{5}}
	\put(92.5,105){{\tiny 2}}
	\put(135,100){\circle{5}}
	\put(130,105){{\tiny n-2}}
	\put(180,100){\circle{5}}
	\put(175,105){{\tiny n-1}}	
	\put(225,100){\circle{5}}
	\put(222.5,105){{\tiny n}}
	\put(107.5,100){$\ldots$}
	\put(0,97.5){$\mathbb A_{n}:$}
	\put(250,97.5){$(n\geq 1)$}	
	\put(52.5,60){\line(1,0){40}}
	\put(137,60){\line(1,0){40}}
	\put(181.5,61){\line(5,2){30}}
	\put(181.5,59.6){\line(5,-2){30}}
	\put(50,60){\circle{5}}
	\put(47.5,65){{\tiny 1}}
	\put(95,60){\circle{5}}
	\put(92.5,65){{\tiny 2}}	
	\put(134.5,60){\circle{5}}
	\put(129,65){{\tiny n-3}}
	\put(179.5,60){\circle{5}}
	\put(174,66){{\tiny n-2}}	
	\put(214,74){\circle{5}}
	\put(208,79){{\tiny n-1}}
	\put(214,47){\circle{5}}
	\put(212,52){{\tiny n}}
	\put(107.5,60){$\ldots$}
	\put(0,57.5){$\mathbb D_{n}:$}
	\put(240,57.5){$(n\geq 4)$}
	\put(52.5,20){\line(1,0){40}}
	\put(97,20){\line(1,0){35}}
	\put(137,20){\line(1,0){35}}
	\put(187,20){\ldots}
	\put(213,20){\line(1,0){35}}
	\put(134.5,17.5){\line(0,-1){18}}
	\put(50,20){\circle{5}}
	\put(47.5,25){{\tiny 1}}
	\put(95,20){\circle{5}}
	\put(92.5,25){{\tiny 2}}
	\put(134.5,20){\circle{5}}
	\put(132,25){{\tiny 4}}
	\put(174.5,20){\circle{5}}
	\put(172,25){{\tiny 5}}
	\put(210.5,20){\circle{5}}
	\put(204,25){{\tiny n-1}}
	\put(250.5,20){\circle{5}}
	\put(248,25){{\tiny n}}
	\put(134.5,-3){\circle{5}}
	\put(138,-5){{\tiny 3}}
	\put(0,17.5){$\mathbb E_{n}:$}
	\put(275,17.5){$(n=6, 7, 8)$}
$
\end{picture}
\caption{Simply laced Dynkin diagrams} 
\label{Fig:ADE_graphs} 
\end{figure}

\noindent
Note that the graphs $\mathbb{E}_n$ are defined in general for all $n\geq 3$,
where $\mathbb{E}_3=\mathbb{A}_2\oplus\mathbb{A}_1$ and for $n\geq 4$ are defined as shown in \cref{Fig:ADE_graphs}.
The graphs $\mathbb{E}_n$ where studied extensively in \cite{gross} where their Coxeter polynomials
were completely factored into cyclotomic and Salem polynomials. The Coxeter polynomials of the $ADE$ graphs are well known
and have been calculated many times (see for instance
\cite{moody,bourbakifourtosix,damianou2014,gross,simson_a_coxeter-gram_classification,simson_a_framework_coauthor,steko}).
One of the main aims of this paper is to find a universal formula for the Coxeter polynomials of a family of
trees which we denote by $\Salem{i}{p_1}{\ldots}{p_k}$. For specific values of $i, k, p_1, \ldots, p_k \in \mathbb N$
we obtain the $ADE$ graphs.

To define the trees $\Salem{i}{p_1}{\ldots}{p_k}$, we recall that
the join of the simple graphs $\Gamma_1,\Gamma_2,\ldots,\Gamma_k$, with a fixed vertex $v_i$
in each of the graphs $\Gamma_i$, is the graph obtained by adding
a new vertex and joining that to $v_i$ for all $i=1,2,\ldots,k$ (see \cite{steko}).

For $k\in\mathbb N, p_1,\ldots,p_k\in\mathbb N$ and $i\in\{0,1,2,\ldots,k\}$,
we define the tree $\Salem{i}{p_1}{\ldots}{p_k}$
to be the join of the Dynkin diagrams
$\mathbb D_{p_1}, \ldots,\mathbb D_{p_i}$ and $\mathbb A_{p_{i+1}}, \ldots$,
$\mathbb A_{p_k}$, 
on their vertices numbered $1$, as shown in \Cref{Fig:T_{p,q,r}}.

The trees $\Salem{0}{p_1}{\ldots}{p_k}$ are the stars $\mathbb T_{p_1-1,\ldots,p_k-1}$ defined in
\cite{ringel_tame_algebras}, which are the join of the Dynkin diagrams $\mathbb A_{p_1-1},\ldots,\mathbb A_{p_k-1}$.
These are also the wild stars defined in \cite{lakatos}.

To the best of my knowledge the graphs $\Salem{i}{p_1}{\ldots}{p_k}$
for $i\geq1$ are defined here for the first time.
For particular values of $i$ and $p_j$, we get some well-known trees.
For example, for $k=2, i=0, p_1=1, p_2=n-2$ we obtain the Dynkin diagrams $\mathbb A_n$, 
for $k=3, i=0, p_1=1, p_2=1, p_3=n-3$ we obtain the Dynkin diagrams $\mathbb D_n$,
for $k=3, i=0, p_1=1, p_2=2, p_3=n-4$ we obtain the diagrams $\mathbb E_n$
and for $k=3, i=1, p_1=n-2, p_2=p_3=1$  we obtain the Euclidean Dynkin diagrams
$\widetilde{\mathbb D}_n$ (see \Cref{Fig:Euclidean_graphs}).
Note that $\Salem{0}{1}{2}{6}=\mathbb{E}_{10}$ and the Coxeter polynomial
$\cox{\mathbb E_{10}}=t^{10}+t^9-t^7-t^6-t^5-t^4-t^3+t+1$ is the well known Lehmer's polynomial which
is conjectured to have the smallest Mahler measure among the monic
integer non-cyclotomic polynomials (see \cite{smythsurveyonmahlermeasure}).

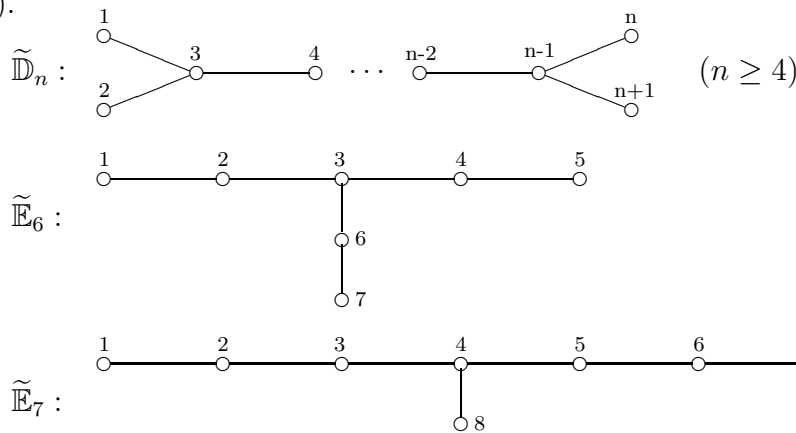
\begin{figure}[ht]
\hspace{4em}
\begin{picture}(355,140)
$
	\put(82.5,130){\line(1,0){40}}
	\put(167,130){\line(1,0){40}}
	\put(212,131){\line(5,2){30}}
	\put(212,129){\line(5,-2){30}}
	\put(77.5,131){\line(-5,2){30}}
	\put(77.5,129){\line(-5,-2){30}}
	\put(80,130){\circle{5}}
	\put(77.5,135){{\tiny 3}}
	\put(125,130){\circle{5}}
	\put(122.5,135){{\tiny 4}}	
	\put(164.5,130){\circle{5}}
	\put(159,135){{\tiny n-2}}
	\put(209.5,130){\circle{5}}
	\put(204,136){{\tiny n-1}}
	\put(45,144){\circle{5}}
	\put(43,149){{\tiny 1}}
	\put(45,116){\circle{5}}
	\put(43,121){{\tiny 2}}
	\put(244.5,144){\circle{5}}
	\put(242,149){{\tiny n}}
	\put(244.5,116){\circle{5}}
	\put(238,121){{\tiny n+1}}
	\put(137.5,130){$\ldots$}
	\put(10,127.5){$\widetilde{\mathbb D}_{n}:$}
	\put(270,127.5){$(n\geq 4)$}
	\put(92.5,90){\line(1,0){40}}
	\put(47.5,90){\line(1,0){40}}
	\put(137.5,90){\line(1,0){40}}
	\put(182.5,90){\line(1,0){40}}
	\put(135,88){\line(0,-1){18}}
	\put(135,65){\line(0,-1){18}}
	\put(45,90){\circle{5}}
	\put(43,95){{\tiny 1}}
	\put(90,90){\circle{5}}
	\put(88,95){{\tiny 2}}
	\put(135,90){\circle{5}}
	\put(132,95){{\tiny 3}}
	\put(180,90){\circle{5}}
	\put(178,95){{\tiny 4}}
	\put(225,90){\circle{5}}
	\put(223,95){{\tiny 5}}
	\put(135,67){\circle{5}}
	\put(140,65){{\tiny 6}}
	\put(135,44){\circle{5}}
	\put(140,42){{\tiny 7}}
	\put(10,72){$\widetilde{\mathbb E}_{6}:$}
	\put(92.5,20){\line(1,0){40}}
	\put(47.5,20){\line(1,0){40}}
	\put(137.5,20){\line(1,0){40}}
	\put(182.5,20){\line(1,0){40}}
	\put(227.5,20){\line(1,0){40}}
	\put(272.5,20){\line(1,0){40}}
	\put(180,18){\line(0,-1){18}}
	\put(45,20){\circle{5}}
	\put(43,25){{\tiny 1}}
	\put(90,20){\circle{5}}
	\put(88,25){{\tiny 2}}
	\put(135,20){\circle{5}}
	\put(132,25){{\tiny 3}}
	\put(180,20){\circle{5}}
	\put(178,25){{\tiny 4}}
	\put(225,20){\circle{5}}
	\put(223,25){{\tiny 5}}
	\put(270,20){\circle{5}}
	\put(268,25){{\tiny 6}}
	\put(315,20){\circle{5}}
	\put(313,25){{\tiny 7}}
	\put(180,-3){\circle{5}}
	\put(185,-5){{\tiny 8}}
	\put(10,2){$\widetilde{\mathbb E}_{7}:$}
$
\end{picture}
\caption{The Euclidean diagrams $\widetilde{\mathbb D}_{n}, \widetilde{\mathbb E}_{6}$ and $\widetilde{\mathbb E}_{7}$} 
\label{Fig:Euclidean_graphs} 
\end{figure}

Let $p(t)$ be a monic polynomial with integer coefficients. We denote the set of its roots
$\{z\in\mathbb C:p(z)=0\}$ by $Z(p(t))$ and  
the maximum value of the set $\{|z|:z\in Z(p)\}$ by $\rho(p(t))$.
For example, for the polynomials $\cox{\mathbb A_n},\cox{\mathbb D_n}$
we have $\rho(\coxwo{\mathbb A_n})=\rho(\coxwo{\mathbb D_n})=1$
while for the polynomials $\cox{\mathbb E_n}$ for $n\geq 10$
we have $\rho(\coxwo{\mathbb E_n})>1$ (see \cite{gross} and \cite{Mckee}).

Assuming that the polynomial $p(t)$ is irreducible then,
if all of its roots lie on the unit circle (or equivalently $\rho(p(t))=1$), it is called a \textit{cyclotomic polynomial}.

Assuming now that the polynomial $p(t)$ is irreducible, non-cyclotomic with only one root outside the unit circle then,
if it has at least one root on the unit circle it is called a \textit{Salem polynomial}
while if it has no roots on the unit circle it is called a \textit{Pisot polynomial} (see \cite{Mckee}).

It is not hard to see that cyclotomic and Salem polynomials are self-reciprocal.
This follows from the following facts.
The polynomial $p(t)$ of degree $n$ is irreducible if and only if the polynomial
$p^*(t):=t^np(\frac{1}{t})$, which we call the reciprocal of $p(t)$, is irreducible.
If $\alpha$ lies on the unit circle then $\alpha$ is a root of $p(t)$
if and only if $\frac{1}{\alpha}$ is also a root of $p(t)$.

We recall from \cite{Mckee} the following definition.

\begin{defin}
(a) A tree $\cal{T}$ is said to be \textit{cyclotomic} if all roots of
the Coxeter polynomial $\cox{\cal{T}}$ are on the unit disk or equivalently
$\cox{\cal{T}}$ is a product of cyclotomic polynomials.

\noindent
(b) A tree $\cal{T}$ is called a \textit{Salem tree} if the Coxeter polynomial
$\cox{\cal{T}}$ has only one root outside the unit circle or equivalently
$\cox{\cal{T}}$ is a product of one of the Salem polynomials and some cyclotomic polynomials.
\end{defin}

\section{Main results}
\label{Section:Main_results}
In this paper we are mainly concerned with the case $k=3$
(i.e. with the trees $\Salem{i}{p}{q}{r}$) and prove 
four theorems about the Coxeter polynomials $\cox{S^{(i)}_{p_1,\ldots,p_k}}$. 
In \cref{Theorem:Recursive relation for Coxeter polynomials} we present a recursive relation for the
Coxeter polynomials of the trees $\Salem{i}{p_1}{\ldots}{p_k}$  and we use it in
\cref{Theorem:Coxeter polynomials of the graphs} to
find the Coxeter polynomials of the trees $S^{(i)}_{p,q,r}$ for all $i=0,1,2,3$.
In \cref{Theorem:Limit theorem for spectrum radius} we 
show that the limits 
$\lim_{p\to\infty}\rho\left(\coxwo{\Salem{i}{p}{q}{r}}\right)$,
$\lim_{q\to\infty}\rho\left(\coxwo{\Salem{i}{p}{q}{r}}\right)$ and
$\lim_{r\to\infty}\rho\left(\coxwo{\Salem{i}{p}{q}{r}}\right)$
are Pisot numbers. We also show that
$$
\lim_{p,q,r\to\infty}\rho\left(\coxwo{\Salem{i}{p}{q}{r}}\right)=2, \text{ for all } i=0,1,2,3.
$$
It was shown by Lakatos \cite{lakatos} that
$$\lim_{p_1,\ldots,p_k\to\infty}\rho\left(\coxwo{\Salem{0}{p_1}{\ldots}{p_k}}\right)=k-1, \text{ for } k\in\mathbb N.
$$
In \cref{Theorem:General limit theorem for spectrum radius} we generalize that result by
showing that
$$
\lim_{p_1,\ldots,p_k\to\infty}\rho\left(\coxwo{\Salem{i}{p_1}{\ldots}{p_k}}\right)=k-1,
\text{ for all } i\in\{0,1,\ldots,k\}.
$$
We mention here that the multiple limits $\lim_{p_1,\ldots,p_i\to\infty}\alpha_n$
are the iterated limits $\lim_{p_1\to\infty}(\ldots(\lim_{p_i\to\infty}\alpha_n))$.

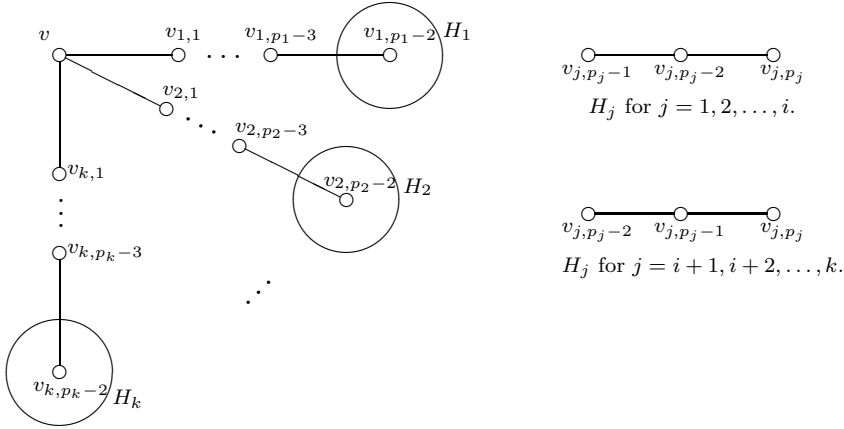
\begin{figure}[ht]
\hspace{4em}
\begin{picture}(315,150)
\put(2,138){{\tiny$v$}}
\put(10,132){\circle{5}}
\put(12.5,132){\line(1,0){40}}
\put(50,139){{\tiny$v_{1,1}$}}
\put(55,132){\circle{5}}
\put(65,131){$\ldots$}
\put(80,139){{\tiny$v_{1,p_1-3}$}}
\put(90,132){\circle{5}}
\put(92.5,132){\line(1,0){40}}
\put(125,139){{\tiny$v_{1,p_1-2}$}}
\put(135,132){\circle{5}}
\put(135,132){\circle{41}}
\put(155,139){{\tiny$H_1$}}
\put(12.25,130.875){\line(2,-1){36}}
\put(49,117){{\tiny$v_{2,1}$}}
\put(50.5,111.75){\circle{5}}
\put(63,102){$\rdots{-26.6}$}
\put(76.5,102.875){{\tiny $v_{2,p_2-3}$}}
\put(78,97.625){\circle{5}}
\put(80.25,96.5){\line(2,-1){36}}
\put(110,82.625){{\tiny $v_{2,p_2-2}$}}
\put(118.5,77.375){\circle{5}}
\put(118.5,77.375){\circle{41}}
\put(140,80){{\tiny $H_2$}}
\put(85,40){$\rdots{45}$}
\put(10,129.5){\line(0,-1){40}}
\put(13,87){{\tiny$v_{k,1}$}}
\put(10,87){\circle{5}}
\put(9,67){$\vdots$}
\put(13,57){{\tiny$v_{k,p_k-3}$}}
\put(10,57){\circle{5}}
\put(10,54.5){\line(0,-1){40}}
\put(0,05){{\tiny$v_{k,p_k-2}$}}
\put(10,12){\circle{5}}
\put(10,12){\circle{41}}
\put(30,0){{\tiny$H_k$}}

\put(210,132){\circle{5}}
\put(200,125){{\tiny$v_{j,p_j-1}$}}
\put(212.5,132){\line(1,0){30}}
\put(245,132){\circle{5}}
\put(235,125){{\tiny$v_{j,p_j-2}$}}
\put(247.5,132){\line(1,0){30}}
\put(280,132){\circle{5}}
\put(275,125){{\tiny$v_{j,p_j}$}}
\put(210,110){{\tiny$H_j$ for $j=1,2,\ldots,i.$}}

\put(210,72){\circle{5}}
\put(200,65){{\tiny$v_{j,p_j-2}$}}
\put(212.5,72){\line(1,0){30}}
\put(245,72){\circle{5}}
\put(235,65){{\tiny$v_{j,p_j-1}$}}
\put(247.5,72){\line(1,0){30}}
\put(280,72){\circle{5}}
\put(275,65){{\tiny$v_{j,p_j}$}}
\put(200,50){{\tiny$H_j$ for $j=i+1,i+2,\ldots,k.$}}
\end{picture}
\caption{The trees $\Salem{i}{p_1}{\ldots}{p_k}$} 
\label{Fig:T_{p,q,r}} 
\end{figure}

\begin{thm}
\label{Theorem:Recursive relation for Coxeter polynomials}
Let $k,p_1,\ldots,p_k \in \mathbb N$ and $p_1\geq2$. Then
$$
\cox{\Salem{0}{p_1}{\ldots}{p_k}}=
(t+1)\cox{\Salem{0}{p_1-1}{\ldots}{p_k}}-t\cox{\Salem{0}{p_1-2}{\ldots}{p_k}}.
$$
If $k\geq2$ and $p_1\geq3$ then
$$
\cox{\Salem{i}{p_1}{\ldots}{p_k}}=
(t+1)\left[\cox{\Salem{i-1}{p_2}{\ldots}{p_k,p_1-1}}-t\cox{\Salem{i-1}{p_2}{\ldots}{p_k,p_1-3}}\right],
$$
for all $i\in\{1,\ldots,k\}$
\end{thm}

\begin{thm}
\label{Theorem:Coxeter polynomials of the graphs}
(a)For $i\leq2$, the Coxeter polynomial $\cox{S^{(i)}_{p,q,r}}$
of the tree $S^{(i)}_{p,q,r}$, is given by the formula 
$$
\cox{\Salem{i}{p}{q}{r}}=\frac{(t+1)^i}{t-1}\left[t^{r+2}F^{(i)}_{p,q}(t)-\left(F^{(i)}_{p,q}\right)^*(t)\right],
$$
where
\begin{gather*}
F^{(0)}_{p,q}(t)=t^{p+q}-\cox{\mathbb A_{p-1}}\cox{\mathbb A_{q-1}},\\
F^{(1)}_{p,q}(t)=t^{p+q-2}(t-1)-\left(t^{p-2}+1\right)\cox{\mathbb A_{q-1}} \text{ and }\\
F^{(2)}_{p,q}(t)=t^{p+q-4}(t-1)^2-\left(t^{p-2}+1\right)\left(t^{q-2}+1\right).
\end{gather*}

(b)The Coxeter polynomial $\cox{S^{(3)}_{p,q,r}}$ is given by the formula 
$$
\cox{\Salem{3}{p}{q}{r}}=(t+1)^3\left[t^rF^{(3)}_{p,q}(t)+\left(F^{(3)}_{p,q}\right)^*(t)\right],
$$
where $F^{(3)}_{p,q}(t)=F^{(2)}_{p,q}(t)$.
\end{thm}

\begin{thm}
\label{Theorem:Limit theorem for spectrum radius}
Let $\rho\left(\coxwo{S^{(i)}_{p,q,r}}\right)$ be the spectral radius of the 
Coxeter transfor-mation of $S^{(i)}_{p,q,r}$. Then we have
\begin{enumerate}
\item $\displaystyle \lim_{r\to\infty} \rho\left(\coxwo{S^{(i)}_{p,q,r}}\right)=\rho\left(F^{(i)}_{p,q}(t)\right)$ and
$\rho\left(F^{(i)}_{p,q}(t)\right)$ is a Pisot number for $i=0,1,2$,
\item $\displaystyle \lim_{p\to\infty} \rho\left(\coxwo{S^{(i)}_{p,q,r}}\right)=
\rho\left(F^{(i-1)}_{q,r}(t)\right)$ for $i=1,2,3$,
\item  $\displaystyle \lim_{p,q\to\infty} \rho\left(\coxwo{S^{(i)}_{p,q,r}}\right)=
\rho\left(t^{r+2}-2t^{r+1}+1\right)$ for $i=0,1,2$,
\item  $\displaystyle \lim_{q,r\to\infty} \rho\left(\coxwo{S^{(i)}_{p,q,r}}\right)=
\rho\left(t^{p}-2t^{p-1}-1\right)$ for $i=1,2,3$ and
\item  $\displaystyle \lim_{p,q,r\to\infty} \rho\left(\coxwo{S^{(i)}_{p,q,r}}\right)=2$ for all $i=0,1,2,3$.
\end{enumerate}
\end{thm}
\begin{thm}
\label{Theorem:General limit theorem for spectrum radius}
For $k,p_1,\ldots,p_k\in\mathbb N$ and all $i\in\{0,1,\ldots,k\}$ we have
$$
\lim_{p_1,\ldots,p_k\to\infty}\rho\left(\coxwo{\Salem{i}{p_1}{\ldots}{p_k}}\right)=k-1.
$$
\end{thm}
\begin{rem}
\label{Remark:Dynkin diagrams}

(a) Note that for $i=1$ or $i=3$ the trees $\Salem{i}{p}{q}{r}$ and $\Salem{i}{r}{q}{p}$
are the same and therefore the case $i=3$ in (1) of \cref{Theorem:Limit theorem for spectrum radius}
is given in (2). Similarly the limit $\lim_{p\to\infty} \rho\left(\cox{S^{(0)}_{p,q,r}}\right)$
can be found using the result of (1).
The same holds for the cases of (3) and (4); the double limit
$\lim_{p,q\to\infty} \rho\left(\cox{S^{(3)}_{p,q,r}}\right)$ is obtained from (4) and
$\lim_{q,r\to\infty} \rho\left(\coxwo{S^{(i)}_{p,q,r}}\right)$ from (3).

(b) In \cite{Mckee} it was shown  by James McKee and Chris Smyth that if a noncyclotomic tree is
the join of cyclotomic trees then it is a Salem tree.
The cyclotomic trees were classified in \cite{smith_some_properties_of_the_spectrum_of_a_graph};
they are the subgraphs of the Euclidean diagram
$\widetilde{\mathbb E}_8=\mathbb E_9$ and of the Euclidean diagrams of \Cref{Fig:Euclidean_graphs}
(see also \cite{Mckee,delapena}).
In \cite{Mckee} the Salem trees were classified and they
include the joins of cyclotomic trees which are not cyclotomic.
It follows from this classification that the cyclotomic cases of
the trees $\Salem{i}{p_1}{\ldots}{p_k}$ are those for $k=i=2$ or
$k=3,i=0,p_1=p_2=p_3=2$ or $k=3,i=0,p_1=1,p_2=p_3=3$ or
$k=3,i=0,p_1=1,p_2=2,p_3=5$ and subgraphs of these. For all the other cases,
$\Salem{i}{p_1}{\ldots}{p_k}$ are Salem trees.

(c) We recall that the Mahler measure of a monic integer polynomial $f(t)$
is $M(f)=\prod\{|z|:z\in Z(f(t)), |z|\geq1\}$ (see \cite{smythsurveyonmahlermeasure}).
We can easily see that if $f$ is cyclotomic, Salem or Pisot then its Mahler measure is $M(f)=\rho(f(t))$.
Lehmer's problem asks if we can chose $f$ with Mahler measure arbitrarily close to 1.
Since the polynomials $\cox{\Salem{i}{p_1}{\ldots}{p_k}}$ have at most one root out\-si\-de the
unit circle it follows that their Mahler measure is $\rho(\cox{\Salem{i}{p_1}{\ldots}{p_k}})$.
\cref{Theorem:Coxeter polynomials of the graphs} in connection with \cref{Lemma: Monotonic} can be used to verify Lehmer's conjecture for the
family of the polynomials $\cox{\Salem{i}{p}{q}{r}}$, asserting that the smallest Mahler measure, larger than 1,
is the Mahler measure of the polynomial $\cox{\Salem{0}{1}{2}{6}}=\cox{\mathbb E_{10}}$
(see also \cite{Mckee} and the recent papers \cite{delapena_tubes_in_derived,delapena_onthemahlermeasure}).

\end{rem}
\begin{exa}
For the case of the Dynkin diagrams $\mathbb D_n$,
\cref{Theorem:Coxeter polynomials of the graphs} gives
\begin{eqnarray*}
\cox{\mathbb D_n}&=&\cox{\Salem{0}{1}{1}{n-3}}\\
&=&\frac{1}{t-1}\left(t^{n-1}(t^2-1)+t^2-1\right)=t^n+t^{n-1}+t+1.
\end{eqnarray*}
For the Euclidean diagrams $\widetilde{\mathbb D}_n$,
\cref{Theorem:Coxeter polynomials of the graphs} gives
\begin{eqnarray*}
\cox{\widetilde{\mathbb D}_n}&=&\coxwo{\Salem{1}{n-2}{1}{1}}\\
&=&\frac{t+1}{t-1}\left[t^3(t^{n-2}-t^{n-3}-t^{n-4}-1)+t^{n-2}+t^2+t-1\right]\\
&=&\left(t^{n-2}-1\right)(t-1)\left(t+1\right)^2
\end{eqnarray*}
and for the diagrams $\mathbb E_n$ it gives
$$
\coxwo{\mathbb E_n}=\coxwo{\Salem{0}{1}{2}{n-4}}=\frac{1}{t-1}\left[t^{n-2}(t^3-t-1)+t^3+t^2-1\right].
$$
All these agree with the known formulas of the Coxeter polynomials of the diagrams
$\mathbb D_n, \widetilde{\mathbb D}_n$ and $\mathbb E_n$
(see \cite{damianou2014,gross} and \cite[Proposition 2.3]{simson_a_coxeter-gram_classification}).
\end{exa}
We also prove the following theorem concerning joins of trees. 
\begin{thm}\label{Theorem:Multiplicity theorem}
Let $\cal{T}$ be the join of the trees $\cal{T}^{(1)},\ldots,\cal{T}^{(k)}, \ k\geq2$.
Suppose that $z$ is a root of the polynomial $\cox{\cal{T}^{(i)}}$ with multiplicity $m_i$.
Then $z$ is also a root of the polynomial $\cox{\cal{T}}$ with multiplicity at least
$$
\min\{m-m_i:i=1,2,\ldots,k\}
$$ 
where $m=m_1+m_2+\ldots+m_k$.
\end{thm}
\begin{rem}
(a)According to \cite{steko2} if the common root $z$ of the polynomials
$\cox{\cal{T}_1},\ldots,\cox{\cal{T}_k}$ is $z\neq\pm1$ then its multiplicity $m_i$ is $1$.
Therefore in that case
\cref{Theorem:Multiplicity theorem} gives that $z$ is a root of $\cox{\cal{T}}$ with multiplicity at least $k-1$.
This result was proved in \cite[Theorem 3.1]{gross}. 
For $z=\pm1$ however, $z$ can be a root of $\cox{\cal{T}}$ with multiplicity less than $k-1$.
For example, consider the join $\cal{T}$ of the Euclidean diagrams $\widetilde{\mathbb D}_4$ as shown in
\cref{Fig:Join of D_3 affine}.
The polynomials $\cox{\cal{T}}$ and $\cox{\widetilde{\mathbb D}_4}$ both have $1$ as a root with multiplicity $2$.

(b)Now suppose that $\cal{T}$ is the join 
of the trees $\cal{T}_1, \cal{T}_2$ and $z$ is the common root of the Coxeter polynomials
$\cox{\cal{T}_1}$ and $\cox{\cal{T}_2}$. Then \Cref{Theorem:Multiplicity theorem}
generalizes a theorem due to Kolmykov \cite{steko} (see also \cite[Theorem 1.5]{gross}) asserting that $z$
is a root of the Coxeter polynomial $\cox{\cal{T}}$.
\end{rem}

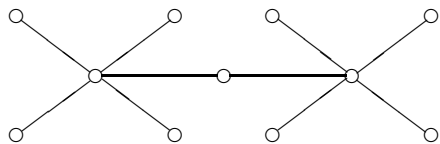
\begin{figure}[ht]
\centering
\begin{picture}(195,55)
\put(20,49){\circle{5}}
\put(80,49){\circle{5}}
\put(50,26){\circle{5}}
\put(21.8,47.2){\line( 4,-3){26}}
\put(78.3,47.3){\line(-4,-3){26.2}}
\put(20,4){\circle{5}}
\put(80,4){\circle{5}}
\put(22.2,5.2){\line( 4, 3){25.6}}
\put(78,5.8){\line(-4, 3){25.4}}
\put(52.2,26.3){\line( 1, 0){43.9}}
\put(98.5,26){\circle{5}}
\put(100.7,26.3){\line( 1, 0){44}}
\put(117,49){\circle{5}}
\put(177,49){\circle{5}}
\put(147,26){\circle{5}}
\put(119,47.5){\line( 4,-3){26.2}}
\put(175.3,47.3){\line(-4,-3){26.4}}
\put(117,4){\circle{5}}
\put(177,4){\circle{5}}
\put(119,5.7){\line( 4, 3){26}}
\put(174.8,5.1){\line(-4, 3){26}}
\end{picture}
\caption{The join of two $\widetilde{\mathbb D}_4$ diagrams} 
\label{Fig:Join of D_3 affine} 
\end{figure}

For the convenience of the reader we include all theorems that will be used, 
in several cases with proofs, thus making this paper self-contained. 
This is done in \cref{Section:Preliminaries}. In
\cref{Section:Proof of theorems} we prove Theorems
\ref{Theorem:Recursive relation for Coxeter polynomials},
\ref{Theorem:Coxeter polynomials of the graphs},
\ref{Theorem:Limit theorem for spectrum radius},
\ref{Theorem:General limit theorem for spectrum radius}
and \ref{Theorem:Multiplicity theorem} formulated in \cref{Section:Main_results}.

\section{Generalities on Coxeter polynomials}
\label{Section:Preliminaries}
In this section we collect and prove some results that we need in the proof of 
Theorems \ref{Theorem:Coxeter polynomials of the graphs}
\ref{Theorem:Limit theorem for spectrum radius}
\ref{Theorem:General limit theorem for spectrum radius}
and \ref{Theorem:Multiplicity theorem}.

The following proposition is due to Subbotin and Sumin and the proof we present here is taken from \cite{steko}.
\begin{proposition}
\label{Proposition:Cut theorem}
Assume that $\cal{T}=(\cal{T}_0,\cal{T}_1)$ is a tree
and let $e=(v_1,v_2)\in \cal{T}_1$ be a splitting edge of the tree $\cal{T}$ 
that splits it to the trees $\cal{R}=(\cal{R}_0,\cal{R}_1)$ and $\cal{S}=(\cal{S}_0,\cal{S}_1)$. 
Assume that $v_1\in \cal{R}_0$ and $v_2\in \cal{S}_0$. Then
$$
\cox{\cal{T}}=\cox{\cal{R}}\cox{\cal{S}}-t\cox{\tilde{\cal{R}}}\cox{\tilde{\cal{S}}}
$$
where $\tilde{\cal{R}}=(\tilde{\cal{R}}_0,\tilde{\cal{R}}_1),\tilde{\cal{S}}=(\tilde{\cal{S}}_0,\tilde{\cal{S}}_1)$
are the subgraphs of $\cal{R}, \cal{S}$ with the vertex sets
$\tilde{\cal{R}_0}=\cal{R}_0\setminus \{v_1\}$ and $\tilde{\cal{S}_0}=\cal{S}_0\setminus \{v_2\}$.
\end{proposition}
\begin{proof}
We enumerate the vertices of $\cal{R}$ and $\cal{S}$ as
$\cal{R}_0=\{u_1,u_2,\ldots,u_k\}$ and $\cal{S}_0=\{u_{k+1},u_{k+2},\ldots,u_{k+m}\}$, 
where $v_1=u_k$ and $v_2=u_{k+1}$. 
Let $\widehat{e}=\{e_1\ldots,e_{k+m}\}$ be the standard basis for the vector space $\mathbb R^{k+m}$,
and let $V_1$ be the vector subspace of $\mathbb R^{k+m}$ with basis $\widehat{e}_1=\{e_1,e_2,\ldots,e_{k}\}$
and $V_2$ the vector subspace of $\mathbb R^{k+m}$ with basis
$\widehat{e}_2=\{e_{k+1},e_{k+2},\ldots,e_{k+m}\}$. Also let $\sigma_i$ be
the $i$th reflection of $\cal{T}$. Then $\Phi_\cal{R}=\sigma_1\sigma_2\ldots\sigma_k$ 
is a Coxeter tra\-nsfo\-rma\-tion of $\cal{R}, \ \Phi_\cal{S}=\sigma_{k+1}\sigma_{k+2}\ldots\sigma_{k+m}$ 
is a Coxeter tra\-nsfo\-rma\-tion of $\cal{S}$ and $\Phi_\cal{T}=\Phi_\cal{R}\Phi_\cal{S}$ is a Coxeter tra\-nsfo\-rma\-tion of $\cal{T}$.
If $R,S$ are the matrices corresponding to $\Phi_R, \Phi_S$
with respect to the bases $\widehat{e_1}, \widehat{e_2}$ then with respect to the basis $\widehat{e}$
the Coxeter transformation $\Phi_\cal{T}$ corresponds to the matrix
$$
\begin{pmatrix}
R&E_{k1}\\
0_{mk}&I_m
\end{pmatrix}
\cdot\begin{pmatrix}
I_k&0_{km}\\
E_{1k}&S
\end{pmatrix},
$$
where $E_{ij}$ is the matrix with 
all entries zero except the $i,j$ entry which is $1$ and $0_{ij}$ is the $i\times j$ 
zero matrix. The Coxeter polynomial of $\cal{T}$ is then given by 
$$
\cox{\cal{T}}=\det(tI_{k+m}-\Phi_{\cal{T}})=
\det\begin{pmatrix}
tI_k-R-E_{k,k}&-E_{k,1}S\\
-E_{1,k}&tI_m-S
\end{pmatrix}.
$$
Subtracting the $k+1$\textsuperscript{th} row from the $k$\textsuperscript{th} row we 
obtain
$$
\cox{\cal{T}}=
\det\begin{pmatrix}
tI_k-R&-tE_{k,1}\\
-E_{1,k}&tI_m-S
\end{pmatrix}.
$$
Expanding the determinant with respect to the $k$\textsuperscript{th} row 
we deduce that
$$
\cox{\cal{T}}=\cox{\cal{R}}\cox{\cal{S}}-t\cox{\tilde{\cal{R}}}\cox{\tilde{\cal{S}}}.
$$
\end{proof}
The following well-known lemma says that the eigenvalues of a bipartite graph are symmetric around $0$,
see \cite{brualdi2011,Cvetkovic2010}.
\begin{lem}
Let $\Gamma$ be a bipartite graph. If $\lambda$ is an eigenvalue of the adjacency
matrix $\ad_{\Gamma}$ of $\Gamma$ then $-\lambda$ is an eigenvalue of $\ad_{\Gamma}$.
\end{lem}
\begin{proof}
Enumerate the vertices of $\Gamma$ such that its adjacency matrix has the form
$$
\ad_\Gamma=
\begin{pmatrix}
0&B\\
B^T&0
\end{pmatrix}.
$$
Suppose that $\begin{pmatrix}x\\y\end{pmatrix}$ is an eigenvector of $\ad_\Gamma$ with eigenvalue $\lambda$.
Then $\begin{pmatrix}-x\\y\end{pmatrix}$ is an eigenvector of $\ad_\Gamma$ with eigenvalue $-\lambda$.
\end{proof}

The next lemma is due to
Hoffman and Smith (see \cite{HoffmanSmith}).
\begin{lem}
\label{Lemma: Monotonic}
If $k,p_1,\ldots,p_k\in\mathbb N$, $0\leq i\leq k$ and $p_j<p'_j$,
for some $1\leq j\leq k$, then
\begin{enumerate}
\item $\rho\left(\coxwo{\Salem{i}{p_1,\ldots}{p_j}{\ldots,p_k}}\right)
\leq \rho\left(\coxwo{\Salem{i}{p_1,\ldots}{p'_j}{\ldots,p_k}}\right)$ if $j> i$ and
\item $\rho\left(\coxwo{\Salem{i}{p_1,\ldots}{p_j}{\ldots,p_k}}\right)
\geq \rho\left(\coxwo{\Salem{i}{p_1,\ldots}{p'_j}{\ldots,p_k}}\right)$ if $j\leq i$.
\end{enumerate}
Moreover, the equalities hold if and only if the tree $\Salem{i}{p_1,\ldots}{p'_j}{\ldots,p_k}$
is cyclotomic.
\end{lem}
We will also need the following lemma.
\begin{lem}
\label{Lemma:Convergence of roots}
Suppose that $f_n(t)=t^ng(t)+h(t)$ is a sequence of functions such that $g,h$ are continuous,
$f_n(z_n)=0$ for all $n\in\mathbb N$ and that $\displaystyle \lim_{n\to\infty}z_n=z_0$. 
If $|z_0|>1$ then $g(z_0)=0$ while if $|z_0|<1$ then $h(z_0)=0$.
\end{lem}
\begin{proof}
Suppose that $|z_0|>1$. The function $h$ is continuous and
$|g(z_n)|=\frac{|h(z_n)|}{|z_n^n|}$. Therefore $\lim_{n\rightarrow\infty}|g(z_n)|=0$.
Since $|g(z_0)|-|g(z_n)|\leq |g(z_0)-g(z_n)|\xrightarrow[n\to\infty]{}0$,
we conclude that $g(z_0)=0$. The proof for the case $|z_0|<1$ is similar.
\end{proof}

\section{Proof of main theorems}
\label{Section:Proof of theorems}
In this section we prove Theorems
\ref{Theorem:Recursive relation for Coxeter polynomials},
\ref{Theorem:Coxeter polynomials of the graphs},
\ref{Theorem:Limit theorem for spectrum radius},
\ref{Theorem:General limit theorem for spectrum radius}
and \ref{Theorem:Multiplicity theorem}.

\begin{proof}[Proof of \cref{Theorem:Recursive relation for Coxeter polynomials}]
For $p_1\geq2$ we split the tree $\Salem{0}{p_1}{\ldots}{p_k}$ by removing the edge $(v_{1,p_1-1},v_{1,p_1})$
and we apply \cref{Proposition:Cut theorem} to get
\begin{align*}
\cox{\Salem{0}{p_1}{\ldots}{p_k}}=&\cox{\mathbb A_1}\cox{\Salem{0}{p_1-1}{\ldots}{p_k}}-t\cox{\Salem{0}{p_1-2}{\ldots}{p_k}}\\
=&(t+1)\cox{\Salem{0}{p_1-1}{\ldots}{p_k}}-t\cox{\Salem{0}{p_1-2}{\ldots}{p_k}}.
\end{align*}
We used that $\cox{\mathbb A_1}=t+1$ which can be easily verified from the definition of the Coxeter polynomial.

For $k\geq2, p_1\geq3$ and $1\leq i\leq k$ if split the tree $\Salem{0}{p_1}{\ldots}{p_k}$
by removing the edge $(v_{1,p_1-2},v_{1,p_1})$ we end up with $\mathbb A_1$ and the join of $i-1$
Dynkin diagrams of type $\mathbb D_{p_2},\ldots,\mathbb D_{p_i}$ and $k-i+1$
Dynkin diagrams of type $\mathbb A_{p_{i+1}},\ldots,\mathbb A_{p_{k}},\mathbb A_{p_{1}-1}$.
We apply \cref{Proposition:Cut theorem} to the edge $(v_{1,p_1-2},v_{1,p_1})$
to get
\begin{align*}
\cox{\Salem{i}{p_1}{\ldots}{p_k}}=&\cox{\mathbb A_1}\left[\cox{\Salem{i-1}{p_2}{\ldots}{p_k,p_1-1}}-t\cox{\Salem{i-1}{p_2}{\ldots}{p_k,p_1-3}}\right].
\end{align*}
\end{proof}
\begin{proof}[Proof of \cref{Theorem:Coxeter polynomials of the graphs}]
For simplicity of notation, we write $u_j,v_j,w_j$ instead
of $v_{1,j},v_{2,j},v_{3,j}$ respectively.

(a) Applying \cref{Proposition:Cut theorem} to the splitting edge $(v,u_1)$ 
of the tree $\Salem{0}{p}{q}{r}$ we get
\begin{gather*}
\cox{\Salem{0}{p}{q}{r}}=\cox{\mathbb A_p}\cox{\mathbb A_{q+r+1}}-t\cox{\mathbb A_{p-1}}\cox{\mathbb A_q}\cox{\mathbb A_r}.
\end{gather*}
The polynomial $\cox{\mathbb A_n}$ can be easily calculated using \cref{Proposition:Cut theorem}.
It satisfies the recurrence 
$$
\cox{\mathbb A_n}=\cox{\mathbb A_{n-1}}+t\left(\cox{\mathbb A_{n-1}}-\cox{\mathbb A_{n-2}}\right)
$$
and is given by the formula $\cox{\mathbb A_{n}}=t^n+t^{n-1}+\ldots+t+1$.

\noindent
Therefore
\begin{gather*}
(t-1)^3\cox{\Salem{0}{p}{q}{r}}=t^{p+q+r+4}-2t^{p+q+r+3}+t^{p+r+2}+t^{q+r+2}-t^{r+2}+\\
t^{p+q+2}-t^{p+2}-t^{q+2}+2t-1\\
\hspace{13ex}=t^{p+q+r+2}(t-1)-t^{r+2}(t^{q}-1)\cox{\mathbb A_{p-1}}+\\
t^2(t^q-1)\cox{\mathbb A_{p-1}}-t+1
\end{gather*}
and hence we get
\begin{eqnarray*}
(t-1)\cox{\Salem{0}{p}{q}{r}}
&=&t^{r+2}\left(t^{p+q}-\cox{\mathbb A_{p-1}}\cox{\mathbb A_{q-1}}\right)\\
&+&t^2\cox{\mathbb A_{p-1}}\cox{\mathbb A_{q-1}}-1\\
&=&t^{r+2}F^{(0)}_{p,q}(t)-\left(F^{(0)}_{p,q}\right)^*(t).
\end{eqnarray*}
For the proof of $i=1,2$ we use the recurrence relation of \cref{Theorem:Recursive relation for Coxeter polynomials}.
For $i=1$, from \cref{Theorem:Recursive relation for Coxeter polynomials} we get that
\begin{align*}
\cox{\Salem{1}{p}{q}{r}}=&(t+1)\left[\cox{\Salem{0}{p-1}{q}{r}}-t\cox{\Salem{0}{p-3}{q}{r}}\right]\\
=&(t+1)t^{r+2}\left[F^{(0)}_{p-1,q}(t)-tF^{(0)}_{p-3,q}(t)\right]\\
-&(t+1)\left[\left(F^{(0)}_{p-1,q}\right)^*(t)-t\left(F^{(0)}_{p-3,q}\right)^*(t)\right]\\
=&(t+1)t^{r+2}\left[F^{(0)}_{p-1,q}(t)-tF^{(0)}_{p-3,q}(t)\right]\\
-&(t+1)\left[F^{(0)}_{p-1,q}(t)-tF^{(0)}_{p-3,q}(t)\right]^*.
\end{align*}
The last equality holds because of the following fact.
For $m_1\geq m_2 \in\mathbb N$ and two polynomials $f,g$ with degrees $\deg{f}=\deg(g)+m_1$
the reciprocal of the polynomial $f(t)+t^{m_2}g(t)$ is
the polynomial $\left(f(t)+t^{m_2}g(t)\right)^*=f^*(t)+t^{m_1-m_2}g^*(t)$.
Therefore to finish the proof for the case $i=1$ it is enough to show that
\begin{align*}
F^{(1)}_{p,q}(t)=F^{(0)}_{p-1,q}(t)-tF^{(0)}_{p-3,q}(t).
\end{align*}
This is an easy verification:
\begin{gather*}
F^{(0)}_{p-1,q}(t)-tF^{(0)}_{p-3,q}(t)=\\
t^{p+q-2}(t-1)-\frac{t^{p-1}-1}{t-1}\cox{\mathbb A_{q-1}}+
t\frac{t^{p-3}-1}{t-1}\cox{\mathbb A_{q-1}}=\\
t^{p+q-2}(t-1)-(t^{p-2}+1)\cox{\mathbb A_{q-1}}.
\end{gather*}

For $i=2$, by \cref{Theorem:Recursive relation for Coxeter polynomials} we get
\begin{align*}
\cox{\Salem{2}{p}{q}{r}}=&(t+1)\left[\cox{\Salem{1}{q}{p-1}{r}}-t\cox{\Salem{1}{q}{p-3}{r}}\right]\\
=&(t+1)t^{r+2}\left[F^{(1)}_{q,p-1}(t)-tF^{(1)}_{q,p-3}(t)\right]\\
-&(t+1)\left[F^{(1)}_{q,p-1}(t)-tF^{(1)}_{q,p-3}(t)\right]^*
\end{align*}
from which follows that to finish the proof for the case $i=2$ is enough to verify that
\begin{align*}
F^{(2)}_{p,q}(t)=F^{(1)}_{q,p-1}(t)-tF^{(1)}_{q,p-3}(t).
\end{align*}

(b) For the Coxeter polynomial $\cox{\Salem{3}{p}{q}{r}}$ we apply
\cref{Proposition:Cut theorem}
to the edge $(w_{r-2},w_{r})$ to obtain
\begin{gather*}
\cox{\Salem{3}{p}{q}{r}}=(t+1)\cox{\Salem{2}{p}{q}{r-1}}-t(t+1)\cox{\Salem{2}{p}{q}{r-3}}.
\end{gather*}
Therefore 
\begin{eqnarray*}
\frac{t-1}{(t+1)^3}\cox{\Salem{3}{p}{q}{r}}&=&\frac{t-1}{(t+1)^2}\cox{\Salem{2}{p}{q}{r-1}}-
t\frac{t-1}{(t+1)^2}\cox{\Salem{2}{p}{q}{r-3}}\\
&=&t^{r+1}F^{(2)}_{p,q}(t)-\left(F^{(2)}_{p,q}\right)^*(t)-
t^{r}F^{(2)}_{p,q}(t)+t\left(F^{(2)}_{p,q}\right)^*(t)
\end{eqnarray*}
and hence we get
\begin{gather*}
\cox{\Salem{3}{p}{q}{r}}=
(t+1)^3\left[t^rF^{(2)}_{p,q}(t)+\left(F^{(2)}_{p,q}\right)^*(t)\right].
\end{gather*}
\end{proof}
\begin{rem}
\label{Remark:Alternative form of the coxeter polynomials of the salem trees}
(a) For the case $i=1$ we could have applied \cref{Proposition:Cut theorem} to the splitting
edge $(u_{p-2},u_{p})$ and use that $\Salem{0}{p}{q}{r}=\Salem{0}{q}{r}{p}$ to obtain
$$
\cox{S^{(1)}_{p,q,r}}=(t+1)\left[t^pF^{(0)}_{q,r}(t)+\left(F^{(0)}_{q,r}\right)^*(t)\right].
$$
Similarly by noting that the graphs $\Salem{1}{p}{r}{q}, \Salem{1}{p}{q}{r}$ are the same
and that the graphs $\Salem{2}{p}{q}{r}, \Salem{2}{q}{p}{r}$ are the same,
\cref{Proposition:Cut theorem} applied to the splitting edge $(v_{q-2},v_{q})$ gives
$$
\cox{S^{(2)}_{p,q,r}}=(t+1)^2\left[t^pF^{(1)}_{q,r}(t)+\left(F^{(1)}_{q,r}\right)^*(t)\right].
$$
(b) Explicitly the polynomials $F^{(i)}_{p,q}(t)$ are
\begin{align*}
F^{(0)}_{p,q}(t)=&\frac{t^p\left(t^{q+2}-2t^{q+1}+1\right)+t^{q}-1}{(t-1)^2},\\
F^{(1)}_{p,q}(t)=&\frac{t^{p-2}\left(t^{q+2}-2t^{q+1}+1\right)-t^{q}+1}{t-1}\\
=&\frac{t^{q}\left(t^{p}-2t^{p-1}-1\right)+t^{p-2}-1}{t-1},\\
F^{(2)}_{p,q}(t)=&t^{p-2}\left(t^{q}-2t^{q-1}-1\right)-t^{q-2}-1.
\end{align*}
\end{rem}

\begin{proof}[Proof of \cref{Theorem:Limit theorem for spectrum radius} ]

(1) From \cref{Theorem:Coxeter polynomials of the graphs} and
\cref{Lemma:Convergence of roots} it is enough to 
show that the sequence $\left(\alpha_r\right)_{r\in\mathbb N}$ defined by
$\alpha_r=\rho\left(\coxwo{\Salem{i}{p}{q}{r}}\right)$,
is convergent (note that from \cref{Remark:Dynkin diagrams},
$\Salem{i}{p}{q}{r}$ are Salem trees and therefore $\alpha_r>1$ for all $r\in\mathbb{N}$).
It follows from \cref{Lemma: Monotonic} that for $i=0,1,2$ the sequence
$\left(\alpha_r\right)_{r\in\mathbb N}$ is increasing.
Since the polynomial $\cox{\Salem{i}{p}{q}{r}}$ is written as
$\cox{\Salem{i}{p}{q}{r}}=t^{r+2}F(t)+G(t)$ where $F(t)$,
$G(t)$ are monic polynomials,
the sequence $\left(\alpha_r\right)_{r\in\mathbb N}$
is also bounded.
For, if $M$ is large enough such that the polynomials 
$F(t), G(t)$ are positive for all $t\geq M$, then $z<M$ 
for all $z\in Z\left(\coxwo{\Salem{i}{p}{q}{r}}\right)$.
Therefore the sequence $\left(\alpha_r\right)_{r\in\mathbb N}$
is indeed convergent.

We now prove that $\rho\left(F^{(i)}_{p,q}\right)$ is a Pisot number
(cf. Lemma 4.3 in \cite{Mckee}).
Let $\epsilon>0$ be small enough and $r$ be large enough such that
$\rho\left(\coxwo{\Salem{i}{p}{q}{r}}\right)>1+\epsilon$ and 
$\left|t^{r+2}F^{(i)}_{p,q}(t)\right|>\left|\left(F^{(i)}_{p,q}\right)^*(t)\right|$
for every $|t|=1+\epsilon$.
From Rouche's theorem (see \cite{rudin}) it follows that
the polynomial $F^{(i)}_{q,r}(t)$ has only one root,
let us say $z_0$, outside the unit circle. If $z_0$ was a Salem number then
we would have $F^*(z_0)=0$ and therefore
$\operatorname{cox}_{\Salem{i}{p}{q}{r}}(z_0)=0$ for all
large $r$, contrary to \cref{Lemma: Monotonic}.
Therefore $z_0=\rho\left(F^{(i)}_{p,q}(t)\right)$ and
$\rho\left(F^{(i)}_{p,q}(t)\right)$ is a Pisot number.

(2) As in (1) we define the sequence $\left(\beta_p\right)_{p\in\mathbb N}$ by
$\beta_p=\rho\left(\coxwo{\Salem{i}{p}{q}{r}}\right)$. Note that 
from \cref{Lemma: Monotonic}, for $i=1,2,3$, the sequence
$\left(\beta_p\right)_{p\in\mathbb N}$ is decreasing.
From \cref{Remark:Alternative form of the coxeter polynomials of the salem trees} it follows that for $i=1,2$
\begin{equation}
\label{Equation:alternative form of coxeter polynomials}
\cox{S^{(i)}_{p,q,r}}=(t+1)^i\left[t^pF^{(i-1)}_{q,r}(t)+\left(F^{(i-1)}_{q,r}\right)^*(t)\right]
\end{equation}
From \cref{Theorem:Coxeter polynomials of the graphs} and from the fact that
$\cox{\Salem{3}{p}{q}{r}}=\cox{\Salem{3}{q}{r}{p}}$ it follows that
(\ref{Equation:alternative form of coxeter polynomials}) holds for $i=3$ also.
Therefore the sequence $\left(\beta_p\right)_{p\in\mathbb N}$ is bounded and
from \cref{Lemma:Convergence of roots} it converges to
$\rho\left(F^{(i-1)}_{q,r}(t)\right)$.

(3) For $q,r\in\mathbb N$ and $i\in\{0,1,2\}$ we define $\ell^{(i)}_{q,r}=\lim_{p\to \infty}\rho(\coxwo{\Salem{i}{p}{q}{r}})$.
By \cref{Lemma: Monotonic}, $\ell_{q,r}$ is monotonic with respect to $q$
From (1) and (2) of this theorem and from the form of the polynomials $F^{(0)}_{q,r}(t), F^{(1)}_{q,r}(t)$,
the sequence $(\ell^{(i)}_{q,r})_{q\in\mathbb N}$ is bounded and therefore convergent
(note that $\ell^{(i)}_{q,r}$ equals $\rho(F^{(0)}_{q,r}(t))$ or $\rho(F^{(1)}_{q,r}(t)))$.
From \cref{Remark:Alternative form of the coxeter polynomials of the salem trees},
\cref{Lemma:Convergence of roots} and the fact that $\ell_{q,r}>1$ we deduce that
$\lim_{p,q\to \infty}\rho\left(\coxwo{\Salem{i}{p}{q}{r}}\right)=\rho\left(t^{r+2}-2t^{r+1}+1\right)$.

(4) The proof for this case is similar to (3).
For $p,q\in\mathbb N$ and $i\in\{1,2,3\}$ we define $\ell^{(i)}_{p,q}=\lim_{r\to \infty}\rho(\coxwo{\Salem{i}{p}{q}{r}})$.
By \cref{Lemma: Monotonic}, $\ell_{p,q}$ is monotonic with respect to $q$.
From (1) and (2) of this theorem and from the form of the polynomials
$F^{(1)}_{p,q}(t), F^{(2)}_{p,q}(t)$
(see \cref{Remark:Alternative form of the coxeter polynomials of the salem trees}),
the sequence $(\ell^{(i)}_{p,q})_{q\in\mathbb N}$ is bounded and therefore convergent
($\ell^{(i)}_{p,q}$ is equal to $\rho(F^{(1)}_{p,q}(t))$ or $\rho(F^{(2)}_{p,q}(t))$).
From \cref{Lemma:Convergence of roots} and the fact that $\ell_{p,q}>1$ we deduce that
$\lim_{q,r\to \infty}\rho\left(\coxwo{\Salem{i}{p}{q}{r}}\right)=\rho\left(t^{p}-2t^{p-1}-1\right)$.

(5) Case $i=0$ was proved by Lakatos in \cite{lakatos}
and therefore we only consider the cases $i=1,2,3$.
Let $\ell^{(i)}_p=\lim_{q,r\to\infty}\rho \left(\coxwo{\Salem{i}{p}{q}{r}}\right)$.
From (4), $\ell^{(i)}_p=\rho(H(t))$ where $H(t):=t^{p}-2t^{p-1}-1$.
Hence $\lim_{p,q,r\to\infty} \rho\left(\coxwo{\Salem{i}{p}{q}{r}}\right)=\lim_{p\to\infty}\rho(H(t))=2$.
\end{proof}

\begin{proof}[Proof of \cref{Theorem:General limit theorem for spectrum radius} ]
For $i\in\{0,1,\ldots,k-1\}$ we have
$$
\cox{\Salem{i}{p_1}{\ldots}{p_k}}=\frac{t^{p_k+1}F(t)-F^*(t)}{t-1}
$$
where 
$$
F(t)=\cox{\Salem{i}{p_1}{\cdots}{p_{k-1}}}-\cox{\mathbb D_{p_1}(t)}\ldots
\cox{\mathbb D_{p_i}}\cox{\mathbb A_{p_{i+1}}}\ldots\coxwo{\mathbb A_{p_{k-1}}}.
$$
Since the Coxeter polynomials of the trees $\Salem{i}{p_1}{\ldots}{p_k}$
and $\mathbb{D}_{p_j}, \mathbb{A}_{p_j}$ are self-reciprocal
(see \cref{Corollary} (c)) the following relation holds
$$
F^*(t)=\cox{\Salem{i}{p_1}{\cdots}{p_{k-1}}}-t\cox{\mathbb D_{p_1}(t)}\ldots
\cox{\mathbb D_{p_i}}\cox{\mathbb A_{p_{i+1}}}\ldots\coxwo{\mathbb A_{p_{k-1}}}.
$$
\Cref{Proposition:Cut theorem} applied to the splitting edge $(v,v_{k,1})$ yields
\begin{gather*}
\cox{\Salem{i}{p_1}{\ldots}{p_k}}=\cox{\Salem{i}{p_1}{\ldots}{p_{k-1}}}\cox{\mathbb A_{p_k}}-\\
t\cox{\mathbb D_{p_1}(t)}\ldots\cox{\mathbb D_{p_i}}\cox{\mathbb A_{p_{i+1}}}\ldots\coxwo{\mathbb A_{p_{k-1}}}\coxwo{\mathbb A_{p_{k}-1}}\\
=\cox{\Salem{i}{p_1}{\ldots}{p_{k-1}}}\frac{t^{p_k+1}-1}{t-1}-\\
t\cox{\mathbb D_{p_1}(t)}\ldots\cox{\mathbb D_{p_i}}\cox{\mathbb A_{p_{i+1}}}\ldots\coxwo{\mathbb A_{p_{k-1}}}\frac{t^{p_k}-1}{t-1}
\end{gather*}
which is exactly the polynomial $\frac{t^{p_k+1}F(t)-F^*(t)}{t-1}.$

Therefore $ \lim_{p_k\to\infty}\rho\left(\coxwo{\Salem{i}{p_1}{\ldots}{p_k}}\right)=\rho(F).$
Similar formulas hold for $i=k$ and inductively we show that
$$
 \lim_{p_{2},\ldots,p_k\to\infty}\rho\left(\coxwo{\Salem{i}{p_1}{\ldots}{p_k}}\right)=\rho(G)
$$
where the polynomial $G(t)$ is given by
$$
G(t)=
\begin{cases}
t^{p_1}-(k-1)t^{p_1-1}-k+2, \ \ \text{if } i\neq0,\\
t^{p_1+1}-(k-1)t^{p_1}+k-2, \ \ \text{if } i=0.
\end{cases}
$$
Hence 
$$
 \lim_{p_{1},p_{2},\ldots,p_k\to\infty}\rho\left(\coxwo{\Salem{i}{p_1}{\ldots}{p_k}}\right)=k-1.
$$
\end{proof}

\begin{figure}[ht]
\centering 
\begin{picture}(150,140)
\put(12,125){$v$}
\put(20,120){\circle{5}}
\put(22.5,120){\line(1,0){104}}
\put(129,120){\circle{5}}
\put(125,112){$v_1$}
\put(125,140){$\cal{T}^{(1)}$}
\put(130,120){\circle{35}}
\put(22,118){\line(2,-1){91}}
\put(115,71){\circle{5}}
\put(115,70){\circle{35}}
\put(107,62){$v_2$}
\put(110,91){$\cal{T}^{(2)}$}
\put(45,60){$\iddots$}
\put(20,118){\line(0,-1){99}}
\put(20,16){\circle{5}}
\put(20,15){\circle{35}}
\put(8,11){$v_n$}
\put(40,14){$\cal{T}^{(i)}$}
\end{picture}
\caption{The join of the graphs $\cal{T}^{(1)},\ldots,\cal{T}^{(i)}$} 
\label{Fig:Gamma} 
\end{figure}
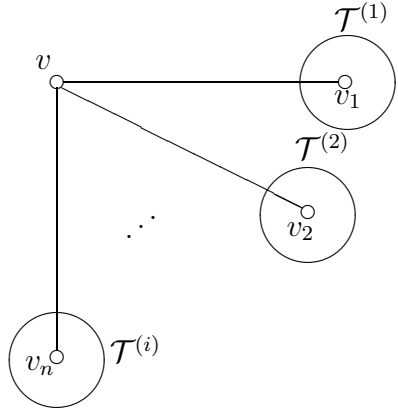

\begin{proof}[Proof of \cref{Theorem:Multiplicity theorem} ]
Let $\cal{T}^{(i)}=(\cal{T}^{(i)}_0,\cal{T}^{(i)}_1)$ where $\cal{T}^{(i)}_0$ is the set of the vertices of $\cal{T}^{(i)}$.
We denote by $\cal{T}^{[i]}$ the join of the graphs $\cal{T}^{(1)},\ldots,\cal{T}^{(i)}$ at the vertices
$v_i\in\cal{T}^{(i)}_0$. The graph $\cal{T}^{(i)}$ looks like the one in \cref{Fig:Gamma}.

Let $i\in\{2,3,\ldots,k\}$. Applying \cref{Proposition:Cut theorem} to the edge $(v,v_i)$ we get
$$
\cox{\cal{T}^{[i]}}=\cox{\cal{T}^{[i-1]}}\cox{\cal{T}^{(i)}}-
t\cox{\cal{T}^{(1)}}\ldots\cox{\cal{T}^{(i-1)}}\cox{\widetilde{\cal{T}^{(i)}}},
$$
where we denote by $\widetilde{\cal{T}^{(i)}}$ the induced subgraph of $\cal{T}^{(i)}$ with the set of vertices
$\widetilde{\cal{T}^{(i)}}_0=\cal{T}^{(i)}_0\setminus\{v_i\}$.

Let us write $P_k(t)=\cox{\cal{T}^{(1)}}\ldots\cox{\widetilde{\cal{T}^{(i)}}}\ldots\cox{\cal{T}^{(k)}}$. Then we have
\begin{align*}
\cox{\cal{T}^{[k]}}=&\cox{\cal{T}^{[k-1]}}\cox{\cal{T}^{(k)}}-tP_k(t)\\
=&\cox{\cal{T}^{[k-2]}}\cox{\cal{T}^{(k-1)}}\cox{\cal{T}^{(k)}}-\\
t\cox{\cal{T}^{(1)}}\ldots&\cox{\cal{T}^{(k-2)}}\cox{\widetilde{\cal{T}^{(k-1)}}}\cox{\cal{T}^{(k)}}-tP_k(t)\\
=&\cox{\cal{T}^{[k-2]}}\cox{\cal{T}^{(k-1)}}\cox{\cal{T}^{(k)}}-t(P_{k-1}(t)+P_k(t))\\
&\ldots\\
=&\cox{\cal{T}^{[0]}}\cox{\cal{T}^{(1)}}\ldots\cox{\cal{T}^{(k)}}-t(P_1(t)+\ldots+P_k(t))\\
=&(t+1)\cox{\cal{T}^{(1)}}\ldots\cox{\cal{T}^{(k)}}-t(P_1(t)+\ldots+P_k(t)).
\end{align*}
Since $z$ is a root of the polynomial $P_i(t)$ of multiplicity $m-m_i$, the theorem follows.
\end{proof}

\vspace{2ex}
\textbf{Acknowledgments:}
I would like to acknowledge the many helpful suggestions of my Ph.D. 
thesis advisor, professor Pantelis Damianou, during the preparation of this paper.
I would also like to thank the anonymous re\-fe\-ree for his constructive comments
and also professor Daniel Simson for
the careful reading of the paper and for his valuable comments and suggestions,
which significantly contributed to improving its quality.

\def\polhk#1{\setbox0=\hbox{#1}{\ooalign{\hidewidth
  \lower1.5ex\hbox{`}\hidewidth\crcr\unhbox0}}}

\end{document}